\documentclass[12pt,english]{smfart}

\usepackage{smfenum}
\usepackage{smfthm}
\usepackage{amssymb}
\usepackage{epic}
\usepackage{eqnarray}

\author{Mathieu Vienney}
\address{Universit\'e de Lyon \\
UMPA ENS Lyon \\
46 all\'ee d'Italie \\
69007 Lyon \\
France}
\email{mathieu.vienney@umpa.ens-lyon.fr}



\usepackage{enumerate}
\usepackage[alphabetic]{amsrefs}
\usepackage[all]{xy}

\newcommand{\Qp}{\mathbb{Q}_p}
\newcommand{\Cp}{\mathbb{C}_p}

\newcommand{\Zp}{\mathbb{Z}_p}

\title[Construction of p-adic Rankin convolutions : the case of positive slope]{A new construction of p-adic Rankin convolutions in the case of
  positive slope}

\newcommand*{\ensqo}[2]%
{\ensuremath{%
    #1/\!\raisebox{-.65ex}{\ensuremath{#2}}}}

\begin{document}

\maketitle

\begin{abstract}

Given two newforms $f$ and $g$ of respective weights $k$ and
$l$ with $k<l$, Hida constructed a $p$-adic $L$-function interpolating the values of
the Rankin convolution of $f$ and $g$ in the critical strip $l \leq s \leq
k$. However, this construction works only if $f$ is an ordinary form. Using
a method developed by Panchishkin to construct $p$-adic $L$-function
associated with modular forms, we generalize this construction to the case
where the slope of $f$ is small.

\end{abstract}

\tableofcontents

\section{Introduction}

In this article, we consider $p$ a prime number and $N$ an integer prime to
$p$. Given two primitive modular forms $f=\sum_{n=1}^\infty a_n q^n$ of
weight $k \geq 2$ for the congruence subgroup $\Gamma_0(N)$, with Dirichlet
character $\psi$, and $g=\sum_{n=1}^\infty b_n q^n$ of weight $l<k$ and
character $\omega$, we can define their Rankin convolution by
$$D_N(s,f,g) = L_N(2s+2-k-l,\psi\omega) L(s,f,g)$$
where
$$L(s,f,g) = \sum_{n=1}^\infty a_n b_n n^{-s}$$

It is a complex function whose analytic properties where studied by Rankin
in the 1930's. Later, the arithmetic properties of its special values were
studied by Shimura in \cite{shimura_special}, proving algebraicity results.

The goal of this article is to construct a $p$-adic function interpolating
the values taken by the Rankin convolution in the critical strip $l \leq s
\leq k-1$.

Several such functions have already been constructed, especially by Hida
(\cite{hida_rankin}) in the case where $a_p$ is a $p$-adic unit (the ordinary
case) or by Vinh Quang My (\cite{my}) in the case of Hilbert modular forms.
Here, we give another construction extending the ordinary case, which we
call the case of positive slope.
Let explain quickly this appellation~: in what follows, we will consider
the p-th Hecke polynomial of $f$ : $X^2 - a_p X + \psi(p) p^{k-1}$ which
we factor in $(X-\alpha)(X-\alpha')$, with $v_p(\alpha) \leq
v_p(\alpha')$. The theory of Newton's polygon implies that the $p$-adic
valuation of $\alpha$ is the smallest slope of the Newton polygon of the
polynomial. In the ordinary case, since $a_p$ is an unit, this slope is
zero. Our construction works even in the case when this slope is non-zero
(but only when $2([v_p(\alpha)]+1) \leq k-l$).

Hida's construction is quite complicated, and cannot be generalized to the
case of positive slope. Here we will ameliorate this method using a general
method developed by Alexei Panchishkin in \cite{pant-new} to construct
$p$-adic $L$ functions associated to modular forms.

This method (explained in details in the second part of this article) has
several steps : 
\begin{enumerate}[1)]
\item the construction of a family of distributions with values in some
  spaces of modular forms (in our case nearly holomorphic modular forms)
\item the application of a projector $\pi_{\alpha}$ onto  a subspace of
  finite dimension, which is the characteristic subspace of the
  Atkin-Lehner operator $U_p$, associated with an eigenvalue $\alpha$.
\item the construction of an admissible measure ``glueing'' the different
  distributions
\item the application of a well-chosen linear form in order to obtain
  scalar-valued distributions.
\end{enumerate}

Using this method, the goal of the third part of this article will be to
prove the following theorem :
\begin{enonce*}{Theorem A} \label{thm_a}
Let $f$, $g$ and $\alpha$ be defined as previously, and $b$ prime to
$Np$.\\ If ${2([v_p(\alpha)]+1)\leq k-l}$, then there exists an unique
$(k-l)$-admissible measure $\tilde{Phi}^{\alpha}$ over $Y = \varprojlim Y_\nu = \varprojlim
(\mathbb{Z}/Np^\nu \mathbb{Z})^\times$ such that for all $\nu \geq 1$ and all
$r \in \{0,1, \dots, k-l-1\}$ : 
$$\int_{a + (Np^\nu)} y_p^r d \tilde{\Phi}^\alpha =(U^\alpha)^{-2\nu}
\pi_{\alpha} \left( U^{2\nu} \sum_{y \in Y_\nu} \psi \bar{\omega}(y) g((y^2
  a)_\nu) E^b_{r,k-l}((y)_\nu) \right)$$
where $g((a)_\nu)$ and $E^b_{r,k-l}((y)_\nu)$ are distributions that will be
defined in the second part.
\end{enonce*}

The hardest point in the proof of this theorem is the verification of two
conditions (especially the one called the divisibility condition) imposed
by Panchishkin's method. We will use techniques developed by Bertrand
Gorsse in his PhD thesis (\cite{gorsse}). \\

In the last chapter, we will compute the values of the Mellin transform of
this measure and link them to the special values of Rankin
convolution. More precisely : 
\begin{enonce*}{Theorem B} \label{calcul_integrale}
Under previous hypothesis over $f,g$ and $\alpha$, for all $\chi
\in \mathbf{Hom_{cont}}(Y,\Cp^\times)^{\text{tors}}$ and $r \in \{0,1,\dots,k-l-1 \}$, 
\begin{multline*}
(-1)^{k+r} l_{f,\alpha} \left(\int_Y \chi(y) y_p^r d\tilde{\Phi}^\alpha(y)\right) = \alpha^{-2\nu}
(1-b^{k-l-2r} \psi \bar{\chi}(b)) \frac{\mathcal{D}_{Np^{2\nu+1}}(l+r,f_0^\rho,g(\chi) \mid_l
  W_{Np^{2\nu+1}})}{\left<f^0,f_0\right>_{Np}} \\
\times \pi^{-l-2r-1} 2^{1-k-l-2r} N^{1-(k-l-2r)/2}p^{(2\nu+1)(1-(k-l-2r)/2)} (-1)^r i^{k-l-1}
\Gamma(l+r) \Gamma(r+1)
\end{multline*}
where $f_0$ is the eigenfunction of $U$ associated with $f$, $f^0 = f_0^\rho
\mid_k W_{Np}$ and $W_{Np^\nu} = \begin{pmatrix} 0 & -1 \\ Np^\nu & 0
\end{pmatrix}$
\end{enonce*}

\section{Notations and generalities}

In this article, we denote by $p$ a fixed prime number, and $\Qp$ the field
of $p$-adic numbers, equipped with its normalized p-adic norm (\emph{i.e.}
$|p|_p = p^{-1}$). We denote by $\Zp$ the ring of $p$-adic integers and the Tate field (the completion of an algebraic closure of $\Qp$)
by $\Cp$, and recall that it is also algebraically closed.

In what follows, it will be useful to consider algebraic numbers as
$p$-adic numbers, so we set once and for all an embedding $i_p :
\bar{\mathbb{Q}} \hookrightarrow \Cp$, and if there is no risk of
confusion write $x$ instead of $i_p(x)$.

\subsection{Classical modular forms}

Here, we just specify the notations we use, for more informations about
holomorphic modular forms, the reader can refer to \cite{shimura} or \cite{miyake}.

As usual, we define the congruence subgroups of $SL_2(\mathbb{Z})$ as
follows : 
\begin{eqnarray*}
\Gamma_0(N)&  =& \left\{ \begin{pmatrix} a & b \\ c & d \end{pmatrix} \in
  SL_2(\mathbb{Z}) | c \equiv 0 \mbox{ mod } N \right\} \\
\Gamma_1(N) & = & \left\{\begin{pmatrix} a & b \\ c & d \end{pmatrix} \in
  SL_2(\mathbb{Z}) | c \equiv 0, a\equiv d \equiv 1 \mbox{ mod } N \right\}
\end{eqnarray*}

For $k$ a positive integer, we define an action of these groups over
functions defined on Poincaré half-plane
$H~=~\left\{z~\in~\mathbb{C}~|~\mathfrak{Im}(z)~>~0~\right\}$ by 
$$f|_k \begin{pmatrix} a & b \\ c & d \end{pmatrix} =(cz+d)^{-k}
f\left(\frac{az+b}{cz+d} \right)$$

For $k,N$ two positive integers, we denote by $\mathcal{M}_k(\Gamma_1(N))$
(resp. $\mathcal{S}_k(\Gamma_1(N))$) the space of holomorphic modular
forms (resp. holomorphic parabolic forms) of weight $k$ for $\Gamma_1(N)$.

Such forms have Fourier developments of the type $f(z)=\sum_{n=0}^\infty a_n
e^{2i\pi n z}$. We will often write $q$ instead of $e^{2i\pi z}$, or
sometimes $e(z)$, where $e$ is the function $z \mapsto e^{2i\pi z}$.

In the case where all the Fourier coefficients of $f = \sum a_n q^n$ are
algebraic, we set the $p$-adic norm of $f$ to be $|f|_p = \sup_{n}
|a_n|_p$. It is a well defined norm.

We define the Petersson inner product over these spaces of modular forms by 
$$<f,g>_{\Gamma_1(N)} = \int_{\Gamma_1(N) \setminus H} f(z) \bar{g}(z) y^{k}
\frac{dx dy}{y^2}$$

If $\chi$ is a Dirichlet character modulo $N$, we extend it in a character of
$\Gamma_0(N)$ by 
$$\chi(\gamma) = \chi(d)\text{ where } \gamma = \begin{pmatrix} a & b \\ c & d
\end{pmatrix} \in \Gamma_0(N)$$
Then, the subspace of $\mathcal{M}_k(\Gamma_1(N))$
(resp. $\mathcal{S}_k(\Gamma_1(N))$) of functions $f$
satisfying $$f|_k \gamma = \chi(\gamma) f, \forall \gamma  \in
\Gamma_0(N)$$ is denoted by $\mathcal{M}_k(\Gamma_0(N),\chi)$ (resp. $\mathcal{S}_k(\Gamma_0(N),\chi)$).

The Hecke operator of index $n$ will be denoted by $T(n)$, and the
matrix of the principal involution of level $N$ by $W_N = \begin{pmatrix} 0 & -1 \\ N & 0
\end{pmatrix}$.

\subsection{Nearly holomorphic modular forms}

As in the previous section, we just specify here the notations for later
use. A good introductory reference about nearly holomorphic modular forms
is \cite{hida}.

Nearly holomorphic modular forms are functions defined over $H$, infinitely
differentiable when considered as functions in two real variables, which
satisfy some automorphic properties (the same as classical modular forms)
and growth conditions.
We denote by $\mathcal{M}_k^r(\Gamma_0(N),\chi)$ the space of nearly
holomorphic modular forms of weight $k$, type $r$, and Dirichlet character
$\chi$ modulo $N$.

Recall that such functions have a Fourier expansion of the type : 
$$f(z) = \sum_{j=0}^r (4 \pi \text{Im}(z))^{-j} \sum_{n=0}^\infty a(n,j,f)e^{2 i
  \pi n z} = \sum_{j=0}^r (4 \pi y)^{-j} \sum_{n=0}^\infty a(n,j,f) q^n$$

In the case when all the Fourier coefficients of $f = \sum_{i=0}^r (4\pi y)^{-i}
\sum_{n=0}^{\infty} a(i,n) q^n$ are algebraic, we define the $p$-adic
norm of $f$ by : $|f|_p = \sup_{i,n}|a(i,n)|_p$. It is a well-defined norm,
extending the one over holomorphic modular forms with algebraic coefficients.

\subsection{Distributions and measures}

Let $Y$ be a profinite topological group, \emph{i.e.} a projective limite
of finite groups : $Y =\displaystyle \varprojlim_{i} Y_i$ (in our case we only need to
consider the case where $Y = \displaystyle \varprojlim_\nu (\ensqo{\mathbb{Z}}{ N p^\nu
\mathbb{Z}})^\times$), $Y$ being equipped with the topology induced by the
product topology, where the $Y_i$ are considered as discrete groups.

\begin{defi}
As usual when $Y=\Zp$, $a + (p)_\nu = (a)_\nu$ will denote the open set
$\left\{ x \in \Zp \mid a \equiv \text{ mod } p^\nu \right\}$, and when $Y
= \varprojlim_\nu (\ensqo{\mathbb{Z}}{Np^\nu \mathbb{Z}})^\times$,
$(a)_\nu$ will denote the open set $\left\{ x \in Y \mid x \equiv a \text{
      mod } Np^\nu \right\}$.
\end{defi}

As in classical analysis, distributions are linear applications over some
function spaces.

\begin{defi}
Let $A$ be a commutative ring. We denote by $\mathbf{Step}(Y,A)$ the space
of locally constant functions over $Y$, with values in $A$. It is easy to
see that such functions can be factored trough one of the $Y_i$.
If $M$ is an $A$-module, we call distribution over $Y$ with values in $M$
a $A$-linear application : $\mu : \mathbf{Step}(Y,A) \rightarrow M$.
In the case where $A$ is equipped with a norm, a distribution is called a measure if it
is uniformly bounded over the open sets of $Y$.

\end{defi}

In this article, we just need to consider a certain class of distributions,
namely admissible measures (which we define only for $Y =
\varprojlim_\nu (\ensqo{\mathbb{Z}}{N p^\nu \mathbb{Z}} )^\times$). A good
reference about admissible measures is \cite{visik}

\begin{defi}
For $h>0$, let $\mathcal{C}^h(Y,\Cp)$ be the space of
locally polynomial functions of the variable $x_p$, of degree strictly less
than $h$ (where $x_p : Y \rightarrow \Cp$ is the canonical projection).
\end{defi}

\begin{rema}
  For $h=1$, it is the space of locally constant functions 
${\mathcal{C}^1(Y,\Cp) =
\mathbf{Step}(Y,\Cp)}$
\end{rema}

\begin{defi}
An $h$-admissible measure is a linear form
 $\mu :\mathcal{C}^h(Y,\Cp) \rightarrow \Cp$ verifying the following growth
 condition : for $r=0,\dots,h-1$
$$\left| \sup_{a \in Y} \int_{a + (Np^\nu)} (x_p - a_p)^r d\mu
  \right|_p = o\left(|p^\nu|_p^{r-h}\right)$$
\end{defi}

Every such linear form can be extended in an unique way in a linear form
over
$\mathcal{C}^{loc-an}(Y,\Cp)$, the space of locally
analytic functions.

For an admissible measure $\mu$, we define its Mellin transform $\mathcal{L}_\mu$ to be the
function defined over $X_p = \mathbf{Hom_{cont}}(Y,\Cp^\times)$ by
$$\mathcal{L}_\mu(x) = \mu(x) = \int_Y xd\mu$$
In fact, it can be proved that this function is always $\Cp$-analytic (over
the $p$-adic analytic group
$\mathbf{Hom_{cont}}(Y,\Cp^\times)$). Moreover, we have the
following important unicity result : 

\begin{theo}\label{unicite-admissible}
Let $\tilde{\Phi}$ be an $h$-admissible measure over $\Zp$, with values in
$\Cp$. Then $\tilde{\Phi}$ is uniquely determined by the numbers: 
$$
L_{\tilde{\Phi}}(\chi(x)x_p^j) = \int_{\Zp} \chi(x)x_p^j d\tilde{\Phi}
\text{, where } j=0,1,\dots,h-1 \text{ and } \chi \in X_p^{\text{tors}}$$
\end{theo}

\section{Panchishkin's method}

In \cite{hida_rankin}, Hida treated the case where $f$ is what he calls an
ordinary form, meaning that its $p$-th Fourier coefficient is a $p$-adic
unit. To that goal, he uses spaces of $p$-adic modular form, \emph{i.e.}
subrings of $\Cp[[X]]$ generated by formal series corresponding to modular
forms of fixed weight, level or character. Then, using inverse limits, he
introduces $p$-adic Hecke algebras acting over these spaces of $p$-adic
modular forms. Especially, he constructs an element $e$ of these Hecke
algebras whcih is idempotent and whose image is of finite dimension. Here
is the point where $f$ needs to be ordinary since in the contrary case
$f\mid e = 0$. The end of its method is quite similar to the one we will
develop later, also using nearly holomorphic modular forms and a similar
linear form.

However, this construction is quite fastidious and only works with ordinary
forms.

The method we expose now was developped by Alexei Panchishkin in
\cite{pant-new}, \cite{pant-bordeaux} and \cite{pant-invent} in order to
construct $p$-adic $L$ functions of modular forms. In a general way, there
are two principal steps : 
\begin{itemize}
\item the construction of distributions with values in finite dimensional
  spaces of nearly holomorphic modular forms
\item the use of a linear form in order to obtain scalar-valued distributions.
\end{itemize}

\subsubsection{Spaces of modular forms}

Let $N'$ be a non-negative integer prime to $p$, $k$ a positive integer and
$r$ a non-negative integer.
Then we denote by $\mathcal{M}_{r,k}(N'p^\nu,\bar{\mathbb{Q}})$ the space
of nearly holomorphic modular forms for $\Gamma_1(N'p^\nu)$, of weight $k$,
of type $r$ and whose coefficients are algebraic.

\begin{defi}
Using previous notations, we define
$$\mathcal{M} = \mathcal{M}_{r,k} = \bigcup_{\nu  \geq 0}
\mathcal{M}_{r,k}(N'p^\nu,\bar{\mathbb{Q}})$$
It is a vector space over $\bar{\mathbb{Q}}$, unfortunately of infinite
dimension.
\end{defi}

This space of modular forms is equipped with a $p$-adic norm given by 
$$\left|\sum_{i=0}^r (4\pi y)^{-i} \sum_{n=0}^\infty a(i,n) q^n \right|_p =
\sup_{i,n} |a(i,n)|_p.$$
and with the Atkin-Lehner operator $U=U_p$ defined as follows
$$g \mid_k U^m = p^{m(k/2-1)} \sum_{u \text{ mod } p^m} g\mid_k
  \begin{pmatrix} 1 & u \\ 0 & p^m \end{pmatrix} = p^{-m} \sum_{u \text{
      mod } p^m} g\left( \frac{z+u}{p^m} \right)$$
It acts over the Fourier coefficients through
$$\left(\sum_{i=0}^r (4\pi y)^{-i} \sum_{n=0}^\infty a(i,n)q^n \right)
\mid_k U^m = \sum_{i=0}^r (4\pi y)^{-i} p^{mi} \sum_{n=0}^\infty a(i,p^m
n)q^n$$
so that it is a $p$-integral operator.

Serre proved in \cite{serre} the following formula, denoting $N'_0 = N'p$
$$U^m = p^{m(k/2-1)}W_{N'_0p^m} \text{Tr}^{N'_0 p^m}_{N'_0} W_{N'_0}$$
implying that
$U^m \left(\mathcal{M}_{r,k}(N'_0 p^m)\right)
\subset \mathcal{M}_{r,k}(N'_0)$

\subsubsection{Projection over finite-dimensional subspaces}

We wish now to work with finite dimensional subspaces of
$\mathcal{M}_{r,k}$.
The very first idea would be to consider the trace operator given by :
$$\text{Tr}^{N'_0p^\nu}_{N'_0} = \sum_{\gamma \in \Gamma_0(N'_0p^\nu) \backslash
  \Gamma_0(N'_0)} f\mid_k \gamma$$

After normalization, it defines a projector given by
$$\left[\Gamma_0(N'_0) : \Gamma_0(N'_0p^\nu) \right]^{-1}
\text{Tr}^{N'_0p^\nu}_{N'_0}$$

However, when $\nu$ grows, the $p$-adic norm of the denominator grows quickly.

Instead of the trace, we will consider projection over finite dimensional
subspaces associated with the operator $U$.

\begin{defi}
For $\alpha \in \bar{\mathbb{Q}}$, we set $\mathcal{M}^\alpha = 
\displaystyle \bigcup_{n\geq 1} \text{Ker}(U-\alpha I)^n$ the characteristic subspace
associated with $\alpha$. We equally denote $\mathcal{M}^\alpha(N'p^\nu) = \mathcal{M}^\alpha \cap
\mathcal{M}_{r,k}(N'p^\nu)$.
\end{defi}

\begin{prop} Let $\alpha$ be a non-zero element of $\bar{\mathbb{Q}}$. Then
\begin{enumerate}[i)]
\item $(U^\alpha)^m : \mathcal{M}^\alpha(N'_0p^m) \to
  \mathcal{M}^\alpha(N'_0p^m)$ is an invertible operator $(m \in
  \mathbb{N})$.
\item The $\bar{\mathbb{Q}}$-vector space $\mathcal{M}^\alpha(N'_0p^m)$
  does not depend over $m \in \mathbb{N}$ and is equal to
  $\mathcal{M}^\alpha(N'_0)$.
\item Let $\pi_{\alpha,m} : \mathcal{M}(N'_0p^m) \to
  \mathcal{M}^\alpha(N'_0p^m)$ be the canonical projector over the
  characteristic subspace associated to $\alpha$, with kernel
$$\text{Ker} \pi_{\alpha,m} = \bigcap_{n \geq 1} \text{Im}(U - \alpha I)^n =
  \bigoplus_{\beta \neq \alpha} \mathcal{M}^{\beta}(N'_0p^m)$$
Then the following diagram is commutative 
$$\xymatrix{
\mathcal{M}(N'_0p^m) \ar[rr]_{\pi_{\alpha,m}} \ar[d]_{U^m} &&
\mathcal{M}^\alpha(N'_0p^m) \ar[d]^{(U^\alpha)^m} \\
\mathcal{M}(N'_0) \ar[rr]_{\pi_{\alpha,0}} && 
\mathcal{M}^\alpha(N'_0) 
}
$$
\end{enumerate}
\end{prop}

\begin{proof}
\begin{enumerate}[i)]
\item The linear operator $(U^\alpha)^m$ acts over a
  $\bar{\mathbb{Q}}$-vector space of finite dimension, with non-zero determinant.
\item Obviously,  $\mathcal{M}^\alpha(N'_0) \subset
  \mathcal{M}^\alpha(N'_0p^m)$. But by i),
$$\mathcal{M}^\alpha(N'_0)\subset \mathcal{M}^\alpha(N'_0p^m) =
 U^m\left(\mathcal{M}^\alpha(N'_0p^m) \right) \subset \mathcal{M}^\alpha(N'_0)$$
\item Being in finite dimension, it is a well-knowm fact that the
  projector $\pi_{\alpha,m}$ is a polynomial in $U$, so commutes with $U$.
Moreover, the restriction of $\pi_{\alpha,m}$ to
$\mathcal{M}(N'_0)$ coincides with $\pi_{\alpha,0}$ because its image is
$$\bigcup_{n \geq 1}  \text{Ker} (U - \alpha I)^n \cap  \mathcal{M}(N'_0) =
\bigcup_{n \geq 1} \text{Ker} (U|_{\mathcal{M}(N'_0)} - \alpha I)^n$$
and its kernel is
$$ \bigcap_{n \geq 1} \text{Im} (U - \alpha I)^n \cap \mathcal{M}(N'_0) =
\bigcap_{n \geq 1} \text{Im}(U|_{\mathcal{M}(N'_0)} - \alpha I)^n$$
\end{enumerate}
\end{proof}

\subsubsection{Modular distributions}

We now consider distributions over $Y = \displaystyle \varprojlim_\nu Y_\nu$ taking
values in $\mathcal{M}$.

The following simple fact, directly coming from the orthogonality of
characters will be very useful later :

\begin{prop}
Let $\Phi$ be a distribution over $Y$ and $a \in Y_\nu$. Then
$$\Phi(a+(Np^\nu)) = \frac{1}{\varphi(Np^\nu)} \sum_{\chi \text{ mod }
  Np^\nu} \chi(a)^{-1} \Phi(\chi)$$
where $\varphi$ is Euler's totient function, and the sum is taken over all
the Dirichlet characters modulo $Np^\nu$.
\end{prop}


\begin{defi}
Let $\Phi$ be a distribution over $Y$ with values in $\mathcal{M}$ and
$\alpha$ a non-zero eigenvalue of $U$ over $\mathcal{M}$. We define
the $\alpha$-primary part $\Phi^\alpha$ of $\Phi$ by
$$\int_Y \varphi d\Phi^\alpha = \pi_{\alpha,0}(\Phi(\varphi)) =
(U^\alpha)^{-\nu} \pi_{\alpha,0} \left(\left(\int_Y \varphi d\Phi \right)
 \mid_k U^\nu \right) \in \mathcal{M}^\alpha$$
for every step function $\varphi$ in
$\mathbf{Step}(Y,\bar{\mathbb{Q}})$ and $\nu$ such that $\int_Y \varphi
d\Phi \in \mathcal{M}(N'_0p^\nu)$.
\end{defi}

\begin{rema}
The third assertion of the previous proposition proves that this definition
does not depend over sufficiently large $\nu$.
\end{rema}

\begin{theo}
If $|\alpha|_p = 1$ (\emph{i.e.} if $\alpha$ is a $p$-adic unit), then
$\Phi^\alpha$ is a measure over $Y$ with values in $\mathcal{M}^\alpha$.
\end{theo}

\begin{proof}
It is sufficient to prove that the distribution $\Phi^\alpha$ is in fact a
measure, that is that there exists a constant $C > 0$ such that for every
open set $a + (Np^m)$, $|\Phi^\alpha(a + (Np^m))|_p \leq C$. 
By hypothesis, there exists $m'$ such that 
$$\Phi(a+(Np^m)) \in \mathcal{M}(N'p^{m'+1})$$
Then the $\alpha$-primary part of $\Phi$ is 
$$\Phi^{\alpha}(a+(Np^m)) = (U^\alpha)^{-m'} \pi_{\alpha,0}(\Phi(a+(Np^m))
| U^{m'})$$

Over the subspace $\mathcal{M}^\alpha \subset \mathcal{M}$, $U^\alpha
= \alpha I + Z$ where $Z$ is a $p$-integral nilpotent operator. 
Thereby for $f \in \mathcal{M},$ $|f|U|_p~\leq~|f|_p$
and $|f|Z|_p \leq |f|_p$.

Next, all the functions $$\Phi(a+(Np^m))^\alpha =\Phi^\alpha(a + (Np^m))= \alpha^{-m'}(\alpha (U^\alpha)^{-1})^{m'}
\pi_{\alpha,0}(\Phi(a+ (Np^m))|U^{m'})$$ are uniformly bounded since $|\alpha^{-1}|_p=1$ and
$$(\alpha (U^\alpha)^{-1})^{m'} = (\alpha^{-1} U^\alpha)^{-m'} = (I +
\alpha^{-1} Z)^{-m'} = \sum_{j=0}^{n-1} { {-m'} \choose j} \alpha^{-j} Z^j$$
where $n$ is the dimensions of $\mathcal{M}^\alpha$ over  $\bar{\mathbb{Q}}$.

Binomial coefficients ${-m'} \choose
j$ being $p$-adic integers, this  proves that the
sum is bounded.
\end{proof}

Now, we consider a finite family of distributions $\Phi_j :
\mathbf{Step}(Y,\bar{\mathbb{Q}}) \rightarrow \mathcal{M}, (j=0,1,
\dots,r^*)$ and we prove an important theorem which will allow us to
``glue'' together the distributions $\Phi_j$ (which are not necessarily
bounded) to an admissible measure.

\begin{theo} Let $\alpha$ be a non-zero eigenvalue of $U$, with
  $0<|\alpha|_p < 1$, and let $h = [v_p(\alpha)] +1$. Assume that there
  exists an integer $\kappa \in \mathbb{N}^*$ such that $\kappa h \leq
  r^*+1$ and verifying the two following conditions :
\begin{enumerate}[i)]
\item Level condition : for all $j \in \{0,1,\dots,r^*\}$, all $\nu \geq 1$
  and all $a \in Y_\nu$
$$\Phi_j(a + (Np^\nu)) \in \mathcal{M}(N'_0p^{\kappa \nu})$$
\item Divisibility condition : there exists $C>0$ such that for all $t\in \{0,1,\dots,r^* \}$, all
  $\nu \geq 1$ and all $a \in Y_\nu$ 
$$\left| U^{\kappa \nu} \sum_{j=0}^{t} {t \choose j} (-a_p)^{t-j}
  \Phi_j(a+(Np^\nu)) \right|_p \leq C p^{-\nu t}$$
\end{enumerate}
Then the linear map $\tilde{\Phi}^\alpha : \mathcal{C}^{r^*+1} (Y,\bar{\mathbb{Q}})
\rightarrow \mathcal{M}^\alpha$ defined by
$$\int_{a + (Np^\nu)} y_p^j d \tilde{\Phi}^\alpha =
\pi_{\alpha,0}\left(\Phi_j(a+(Np^\nu))\right)$$
for all $j \in \{0,1,\dots,r^*\}$ is a $r^* +1$-admissible measure.
\end{theo}

\begin{proof}
To prove that the linear map we consider is an admissible measure, it is
sufficient to prove the growth condition defining admissible measures, that
is to say that for all $a \in Y$ and all $t\in \{0,1,\dots,r^*\}$
$$\left|\int_{a + (Np^\nu)} (y_p - a_p)^t d\tilde{\Phi}^\alpha \right|_p =
o(p^{\nu(r^*+1 -t)}) \text{ as } \nu\rightarrow + \infty$$
On the one hand, over the subspace $\mathcal{M}^\alpha(N'_0)$, we have  $U = \alpha I + Z$ with
$Z$ a nilpotent operator, verifying $Z^n=0$ where $n$ is the dimension of
$\mathcal{M}^\alpha(N'_0) =
\mathcal{M}^\alpha(N'_0 p^\nu)$ over $\bar{\mathbb{Q}}$.

On the other hand, by the definition of $\tilde{\Phi}^\alpha$ : 
\begin{multline*}
\int_{a+(Np^\nu)} (y_p - a_p)^t d \tilde{\Phi}^\alpha = \sum_{j=0}^t {t
  \choose j} (-a_p)^{t-j} \pi_{\alpha,0}(\Phi_j(a+(Np^\nu))) \\
= \alpha^{-\kappa \nu } \alpha^{\kappa \nu} U^{-\kappa \nu} \left[
\pi_{\alpha,0} U^{\kappa \nu} \left( \sum_{j=0}^t {t \choose j}
  (-a_p)^{t-j} \Phi_j(a+(Np^\nu)) \right) \right]
\end{multline*}

Moreover, by the same argument than previously, the operators
$$\alpha^{\kappa \nu} U^{-\kappa \nu} = (\alpha^{-1} U)^{- \kappa \nu} = (I
+ \alpha^{-1} Z)^{- \kappa \nu} = \sum_{i=0}^{n-1} { - \kappa \nu \choose
  i} (\alpha^{-1} Z)^i$$
are uniformly bounded by a constant $C_1>0$, so by the divisibility condition 
$$\left| \int_{a + (Np^\nu)} (y_p - a_p)^t d \tilde{\Phi}^\alpha \right|_p
\leq C C_1 |\alpha|_p^{- \nu \kappa} |p^\nu|^t = o(p^{\nu(r^*+1 -t)})$$
when $\nu \rightarrow + \infty$ because $|\alpha|_p = p^{-v_p(\alpha)},
v_p(\alpha) < h$, so $p^{\kappa \nu v_p(\alpha)} = o(p^{\nu \kappa h}) =
o(p^{\nu (r^*+1)})$.

\end{proof}

\subsubsection{Eigenfunctions of Atkin-Lehner's operator}

We now introduce some basic notions about eigenfunctions of the operator $U$.

\begin{defi}
A primitive cusp form $f = \displaystyle \sum_{n\geq 0} a_n q^n \in
\mathcal{S}_k(\Gamma_0(N),\psi)$ is said to be associated with the
eigenvalue  $\alpha$ if there exists a cusp form $f_0 = \displaystyle \sum_{n=1}^\infty a(n,f_0) q^n$
such that $f_0 \mid U = \alpha f_0$ and $f_0 \mid T(l) = a_l f_0$ for every $l
\nmid Np$.
\end{defi}

In fact, if $f = \sum_{n=1}^\infty a_n q^n$ is a primitive cusp form, it is
necessarily associated with an eigenvalue.\\
Let $X^2 - a_p X + \psi(p)p^{k-1}$ be the $p$-th Hecke polynomial
associated with this form, factored into $(X - \alpha)(X - \alpha')$.
Let $f_0 = f_{0,\alpha}$ be defined by $f_0 = f - \alpha' V(f)$ where $V$ acts through $V\left( \sum_{n=0}^\infty a_n q^n \right) =
\sum_{n=0}^\infty a_n q^{pn}$

\begin{prop}
Under the previous hypothesis $f_0 \mid U = \alpha f_0$.
\end{prop}

\begin{proof}
The Fourier expansion of $f_0$ is 
$$\sum_{n=1}^\infty a(n,f_0) q^n = \sum_{n=1}^\infty a_n q^n - \alpha'
\sum_{n=1}^\infty a_n q^{pn}$$
The expansion of $f_0 \mid U = f \mid U - \alpha' f$ is 
$$\sum_{n=1}^\infty a_{pn} q^n - \alpha' \sum_{n=1}^\infty a_n q^n$$

To be able to link these two functions, we need to know more about their
Fourier coefficients, and so about those of $f$. $f$ being primitive,
we already know that  $f \mid T(p) = a_p f$. But a well-known formula
asserts that the $n$-th Fourier coefficient $b_n$ of $f \mid T(p)$ is given
by
\begin{eqnarray*}
b_n = \begin{cases} a_{pn} \text{ si } (n,p)=1 \\
a_{pn} + \psi(p)p^{k-1} a_{\frac{n}{p}} \text{ si } (n,p)=p
\end{cases}
\end{eqnarray*}

If $(n,p)=1$, then $a_n a_p = a_{np}$. But the  $n$-th coefficient of
$f_0 \mid U$ is $a_{pn} - \alpha' a_n = a_{pn} - (a_p - \alpha) a_n =
\alpha a_n$, which is exactly the $n$-th coefficient of $\alpha f_0$.

In the case where $(n,p)=p$, by what we said before, Fourier coefficients of $f$ verify 
$$a_p a_n = a_{pn} + \psi(p)p^{k-1} a_{\frac{n}{p}}.$$
So $(\alpha + \alpha')a_n = a_{pn} + \alpha \alpha' a_{\frac{n}{p}}$.

Eventually
$$\alpha (a_n - \alpha' a_{\frac{n}{p}}) = a_{pn} - \alpha' a_n$$
so that $ \alpha f_0$ and $f_0 \mid U$ have the same Fourier coefficients,
so are equal.
\end{proof}


There is no difficulty to verify the property about Hecke operators of rank
prime to $Np$, proving the following : 

\begin{prop}
Every primitive cusp form is associated with at least an eigenvalue $\alpha$ of $U$.
Moreover, we can choose for $\alpha$ a root of the $p$-th Hecke polynomial
of $f$.
\end{prop}

\subsubsection{Application of a suitable linear form}
Let $\alpha \in \bar{\mathbb{Q}}$ be a non-zero eigenvalue of $U$
associated with a primitive cusp form
$f~\in~\mathcal{S}_k(\Gamma_0(N),\psi)$ and let
$$f_0 = f_{0,\alpha} = f-\alpha' V(f)$$
be an eigenfunction of $U$ such that $f_0 \mid_k U = \alpha f_0$. Set
$$f^0 = f_0^\rho \mid_k W_{N_0}, f_0^\rho = \sum_{n=0}^\infty
\overline{a_n(f_0)} q^n$$

\begin{prop}
\begin{enumerate}[i)]
\item $U^* = W_{N_0}^{-1} U W_{N_0}$ acting over
  $\mathcal{S}_{r,k}(\Gamma_0(N_0),\psi)$ is the adjoint of the operator $U$
  with respect to Petersson inner product.
\item $f^0 \mid_k U^* = \bar{\alpha} f^0$, and for all ``good'' primes $l$
  (\emph{i.e.} $l \nmid Np$), $T(l)f^0 = a_l(f) f^0$
\item For all $g \in \mathcal{M}_{r,k}(\Gamma_1(N_0))$, we have 
$$\left< f^0,g\right> = \left< f^0,\pi_{\alpha,0}(g) \right>$$
meaning that the linear form $g \mapsto \left<f^0,g\right>$ defined over
$\mathcal{M}_{r,k}(\Gamma_1(N_0))$ is zero over $\text{Ker}(\pi_{\alpha,0})$.
\item If $g \in \mathcal{M}_{r,k}(\Gamma_1(N_0p^\nu),\mathbb{Q})$ and $\alpha\neq 0$,
  then
$$\left<f^0,\pi_{\alpha}(g) \right> = \alpha^{-\nu}
\left<f^0,U^\nu(g)\right>$$
where $\pi_{\alpha}(g) = (U^\alpha)^{-\nu} \pi_{\alpha,0}(U^\nu(g)) \in
\mathcal{M}_{r,k}(\Gamma_1(N_0p^\nu),\mathbb{Q})$ is the canonical
$\alpha$-primary projection of $g$.
\item We set
$$\mathcal{L}_{f,\alpha}(g) = \frac{\left<f^0,\alpha^{-\nu} U^\nu(g)
  \right>}{\left<f^0,f_0\right>} = \frac{\left<f^0,\pi_\alpha(g)
    \right>}{\left< f^0,f_0 \right>}$$
Then it defines a linear form
$$\mathcal{L}_{f,\alpha} :
\mathcal{M}_{r,k}(\Gamma_1(Np^{\nu+1}),\bar{\mathbb{Q}}) \rightarrow
\bar{\mathbb{Q}}$$ 
over $\bar{\mathbb{Q}}$ and there exists a unique $\Cp$-linear form
$l_{f,\alpha} \in \mathcal{M}_{r,k}^\alpha(N_0,\Cp)^*$ acting over the
$\Cp$-generators of the vector space
$\mathcal{M}^\alpha(N_0p^\nu,\Cp) = \mathcal{M}_{r,k}^\alpha(N_0
p^\nu,\bar{\mathbb{Q}}) \otimes_{\bar{\mathbb{Q}}} \Cp$ by
$$l_{f,\alpha}(g) = i_p(\mathcal{L}_{f,\alpha}(g)) = i_p
\left(\frac{\left<f^0,\alpha^{-\nu} U^\nu(g)\right>}{\left<f^0,f_0\right>}
\right)$$
\end{enumerate}
\end{prop}

\begin{proof}
\begin{enumerate}[i)]
\item See \cite{miyake}, theorem 4.5.5.
\item We have
$$f^0 \mid_k U^* = f_0^\rho \mid_k W_{N_0} W_{N_0}^{-1} U W_{N_0} =
\bar{\alpha} f_0^\rho \mid_k W_{N_0} = \bar{\alpha} f^0$$
\item For every function
$$g_1 = (U-\alpha I)^n g \in \text{Ker}(\pi_{\alpha,0}) =
\text{Im}(U-\alpha I)^n$$
we have
$$\left<f^0,g_1 \right> = \left<f^0,(U-\alpha I)^n g \right> = \left<(U^* -
  \bar{\alpha} I)f^0,(U - \alpha I)^{n-1} g\right> = 0$$
so that for $g_1 = g - \pi_{\alpha,0}(g)$, we have
\begin{eqnarray*}
\left<f^0,g \right> & = & \left<f^0,\pi_{\alpha,0}(g) + (g -
  \pi_{\alpha,0}(g))\right> \\
& = & \left< f^0,\pi_{\alpha,0}(g) \right> + \left<f^0,g_1\right> \\
& = & \left< f^0,\pi_{\alpha,0}(g) \right>
\end{eqnarray*}

\item We use the result of ii) : $(U^*)^\nu f^0 = \bar{\alpha}^\nu f^0$ :
\begin{eqnarray*}
\alpha^\nu \left< f^0,\pi_{\alpha}(g)\right> & = & \left<(U^*)^\nu
  f^0,U^{-\nu} \pi_{\alpha,0}(U^\nu(g))\right> \\
& = & \left< f^0,\pi_{\alpha,0}(U^\nu(g))\right> \\
& = &  \left<f^0,U^\nu(g)\right>
\end{eqnarray*}

\item Consider the complex vector space
$$\text{Ker}(L_{f,\alpha}) = \left<f^0\right>^\perp = \{g \in
\mathcal{M}_{r,k}(N_0,\mathbb{C}) \mid \left<f^0,g\right> =0 \}$$
It admits an algebraic basis (\emph{i.e.} consisting of elements with
Fourier coefficients in $\bar{\mathbb{Q}}$) since it is stable under all
``good'' Hecke operators $T(l), l \nmid Np$ : 
$$\left<f^0,g\right> = 0 \Rightarrow \left<f^0,T(l) g \right> =
\left<T(l)|f,g\right> = 0$$
Such a basis is obtained by simultaneous diagonalization of all the $T(l)$.

\end{enumerate}
\end{proof}

The method we will use in what follows is then to construct an admissible
measure with values in (nearly holomorphic) modular forms and then to apply
it the linear form we just constructed to obtain a scalar-valued measure.

\section{Construction of the distributions}

In this section, we construct several modular distributions, to which the
Panchishkin's method can be applied so as to obtain an admissible measure.

As in the introduction, let $f$ be a newform of weight $k \geq 2$, modular for the
congruence subgroup $\Gamma_0(N)$, with $(N,p)=1$, and with Dirichlet
character $\psi$ modulo $N$. Let $g$ be a primitive form of level $l<k$,
modular for $\Gamma_0(N)$ and with Dirichlet character $\omega$ modulo
$N$. We denote their Fourier expansions by
$$f = \sum_{n=1}^\infty a_n q^n \text{ and } g = \sum_{n=1}^\infty b_n q^n$$

We denote by $\alpha$ and $\alpha'$ the roots of the p-th Hecke polynomial
of $f$, that is
$$X^2 - a_p X + \psi(p) p^{k-1} = (X - \alpha)(X - \alpha')$$
and we suppose moreover that $v_p(\alpha) \leq v_p(\alpha')$. Hida treated
the case where $|a_p|_p = 1$ which implies that $v_p(\alpha) = 0$ by the
theory of the Newton polygon. Here, the first slope of the Newton polygon
can be non-zero, and it is the reason why we call it the case of positive slope.
These distributions will be defined on the profinite group
$$Y = \varprojlim_\nu Y_\nu = \varprojlim_\nu \left(
  \ensqo{\mathbb{Z}}{Np^\nu \mathbb{Z}} \right)^\times$$

By the chinese remainder theorem, 
$$Y = \left(\ensqo{\mathbb{Z}}{N \mathbb{Z}}\right)^\times \times
\Zp^\times$$

The distributions we construct are very similar to those used by Hida in
\cite{hida_rankin}, using partial modular forms and Eisenstein distributions.

\subsection{Partial modular forms}

\begin{defi}
Let $g = \sum_{n=0}^\infty a_n q^n \in
\mathcal{M}_l(\Gamma_1(N),\bar{\mathbb{Q}})$, $\nu \in \mathbb{N}$ and $a \in \left(\mathbb{Z}/Np^\nu \mathbb{Z}\right)^\times$.
Then we set
$$g((a)_\nu) = \sum_{n\equiv a \text{ mod } Np^\nu}^\infty a_n q^n$$
\end{defi}

In fact, $g((a)_\nu)$ is a modular form with the same weight than $g$, but
with higher level, as stated by the following proposition.

\begin{prop}
With $g,\nu$ and $a$ as previously, $g((a)_\nu) \in
\mathcal{M}_l(\Gamma_1(N^2 p^{2\nu}))$
\end{prop}

\begin{proof}
It is a special case of the proposition 8.1 of \cite{hida_rankin}.
\end{proof}

Then, it is obvious that we define a distribution over $Y =
\varprojlim \left(\mathbb{Z}/Np^\nu \mathbb{Z} \right)^\times$, with values
in modular forms.

\subsection{Eisenstein distributions}
Here, we quickly redefine Eisenstein distributions, following Panchishkin's paper
\cite{pant-invent} where the reader can find all the details, especially
about the Fourier expansion of such distributions.

First, we need to define classical Eisenstein series :
\begin{defi}
Let $l,N$ be to natural integers, $l,N \geq 1$, $z \in H, s \in
\mathbb{C}$. For ${l+\mathbf{Re}(s)} \geq 2$ and $a,b \in
\mathbb{Z}/N\mathbb{Z}$ we set
$$E_{l,N}(z,s;a,b) = \sum_{(c,d)\equiv (a,b) \text{ mod } N \atop (c,d)
  \neq (0,0)}
(cz+d)^{-l}|c\bar{z}+d|^{-2s}$$
\end{defi}

These functions have well known automorphic properties and their Fourier
expansion is also known.

\begin{defi}
For $0\leq r \leq l-1$, $\nu \geq 0$ and $a \in \left(\mathbb{Z}/Np^\nu
  \mathbb{Z}\right)^\times$ we set
$$E_{r,l}((a)_\nu) = \frac{\Gamma(l-r)(Np^\nu)^{l-2r}}{(-2i\pi)^{l-2r}
  (-4\pi y)^r} \sum_{b \text{ mod } Np^\nu} e \left(-\frac{ab}{Np^\nu}\right) E_{l,Np^\nu}(z,-r;0,b)
$$
\end{defi}

\begin{rema}
Even if the notation $E_{r,l}((a)_\nu)$ is not obvious, it denotes a
function of the complex variable $z$.
\end{rema}

\begin{prop}
Functions defined previously are nearly holomorphic modular
forms. Precisely :
$$E_{r,l}((a)_\nu) \in \mathcal{M}_{r,l}(\Gamma_1(N^2p^{2\nu}))$$
Their Fourier expansion is given by
$$E_{r,l}((a)_\nu) = \varepsilon_{r,l,\nu}(a) + (4\pi y)^{-r}
\sum_{n=1}^\infty \left( \sum_{d\mid n \atop d \equiv a \text{ mod } Np^\nu}
  sgn(d) d^{l-2r-1} \right) W(4\pi n y,l-r,-r) q^n$$
where 
$$\varepsilon_{r,l,\nu}(a) = \frac{1}{2}(-4\pi
y)^{-r} \left[\frac{\Gamma(l+s)}{\Gamma(l+2s)} \zeta(1-l-2s,a,Np^\nu)
  \right] \mid_{s=-r}
$$
with $\zeta(1-l-2s,a,Np^\nu)$ denotes the partial zeta function and $W(4\pi
n y ,l-r,-r)$ is the Witthaker polynomial
$$W(y,\alpha,-r) = \sum_{i=0}^r (-1)^i {r \choose i}
\frac{\Gamma(\alpha)}{\Gamma(\alpha-i)} y^{r-i}$$

\end{prop}

\begin{proof}
See \cite{pant-invent}, section 2.
\end{proof}

\section{Construction of the admissible measure}

In the following we shall use distributions defined over $\Zp^\times$ for which the divisibility condition will be easier to prove. 
The distributions we use are obtained from the previously defined partial
modular distributions and Eisenstein series :

$$g((a)_\nu) = \sum_{0<n \equiv a \text{ mod } Np^\nu} b(n) q^n$$
$$E^b_{r,m}((a)_\nu) = E_{r,m}((a)_\nu) - b^{m-2r} E_{r,m}((b^{-1}a)_\nu)$$
where $b$ is a fixed integer prime to $Np$.

\begin{defi} For $0 \leq r \leq k-l-1$, we define $$\Phi_r((y)_\nu) = (-1)^r \sum_{a \in Y_\nu}
\psi(a) \bar{\omega}(a) g((a^2y)_\nu) E^b_{r,k-l}((a)_\nu)$$
\end{defi}

The normalisation factor $(-1)^r$ will be useful later to verifiy the
congruence condition.

These distributions were suggested in the ordinary case by Panchishkin
\cite{pant-duke}. Instead of using nearly holomorphic modular forms, he was
using their holomorphic projection, but this does not change the final
result because we then use the Petersson inner product with an holomorphic
modular form.

Now, we want to use the Panchishkin method described before, with
$\kappa=2$. For this, we need to verify both conditions : the level
condition and the condition about the divisibility of the Fourier coefficients.

\begin{rema}
Remind that in order to apply Panchishkin's method, we need  to construct admissible measures out of sufficiently many distributions. Here, we defined $k-l$ distributions $\Phi_j,
j=0,\dots,k-l-1$. So everything we do next will be true only under the
hypothesis that $2 ([v_p(\alpha)] +1) \leq k-l$. This means that the slope
of the Newton polygon associated with the Hecke polynomial is not too
big. Especially, this is the case in the ordinary case, \emph{i.e.} when
$\alpha$ is a $p$-adic unit.
\end{rema}

\begin{exem}
Take Ramanujan's $\Delta$ function for $f$, and $g$ of weigth $2$ (for
exemple the modular form associated with an elliptic curve), and
$p=7$. Then, the $p$-th Hecke polynomial of $f$ is $X^2 + 16744 X +
7^{11}$. By the Newton polygon theory, we see that $v_p(\alpha) = 1$ while
$k-l=10$, so we have enough distributions to apply our method.
\end{exem}

\subsection{The level condition}
As usual, the level condition is easier to verify than the divisibility
one. Indeed, we know that if $a \in Y_\nu$, then the partial modular form
is in $\mathcal{S}_l(N^2p^{2\nu})$.

We also know that the Eisenstein series verify $E_{r,k-l}((a)_\nu)
\in \mathcal{M}_{r,k-l}(Np^\nu) \subset \mathcal{M}_{r,k-l}(N^2p^{2\nu})$.

Hence, for all $a,y \in Y_\nu$,  $g((a^2y)_\nu)
E^b_{r,k-l}((a)_\nu) \in \mathcal{S}_{r,k}(N^2p^{2\nu})$ so that $$\Phi_r((y)_\nu) = (-1)^r \sum_{a\in Y_\nu} \psi \bar{\omega}(a) g((a^2
  y)_\nu) E^b_{r,k-l}((a)_\nu) \in \mathcal{M}_k(N^2p^{2\nu})$$
proving the level condition.

\subsection{Using distributions over $\Zp^\times$}

In this part we aim at explaining how and why we reduce the proof to
distributions over $\Zp^\times$ as Gorsse did for symetric squares \cite{gorsse}.

What we want to prove is an inequality with the p-adic norm, so congruences
modulo $p$. But reasonning with the profinite group $Y$ gives congruences
modulo $Np^k$, while using $\Zp^\times$ is the same that manipulating congruences
modulo a power of $p$. \\

Let $\xi$ be a fixed Dirichlet character modulo $N$. We can then define a distribution
$\Phi_r^\xi$  over $\Zp^\times$ by 
$$\int_{\Zp^\times} \chi(x) d \Phi_r^\xi(x) = \int_Y \chi(x_p) \xi(x_N)
d\Phi_r(x)$$
where  $x_p : Y \rightarrow \Zp^\times$ et $x_N : Y \rightarrow
\left(\mathbb{Z}/N \mathbb{Z} \right)^\times$ are the canonical embeddings.

Given $\Phi_r$, it defines the family $\Phi_r^\xi$, where $\xi$ runs through the
set of Dirichlet characters modulo $N$. Conversely, given the family of the
$\Phi_r^\xi$, it is possible to construct $\Phi_r$.\\
Indeed, the characteristical function of the open set $y + (Np^\nu)$ is 
$$\frac{1}{\varphi(Np^\nu)} \sum_{\chi \text{ mod }Np^\nu} \bar{\chi}(y)
\chi = \frac{1}{\varphi(N)} \sum_{\xi \text{ mod } N} \bar{\xi}(y_N) \left(
  \frac{1}{\varphi(p^\nu)} \sum_{\chi \text{ mod } p^\nu}
  \bar{\chi}(y_p) (\chi \times \xi) \right)$$

So
$$\Phi_r(y + (Np^\nu)) = \frac{1}{\varphi(N)} \sum_{\xi \text{ mod } N}
\bar{\xi}(y_N) \left(\frac{1}{\varphi(p^\nu)} \sum_{\chi \text{ mod }
    p^\nu} \bar{\chi}(y_p) \Phi_r^\xi(\chi) \right)$$

We choose to prove the divisibility condition for each family $\Phi_r^\xi,
r=0,1,\dots,k-l-1$ for any fixed $\xi$. Then two solutions are possible to
conclude : 
\begin{itemize}
\item either remark that considering the expression of the $\Phi_r$
  depending on the $\Phi_r^\xi$, the family of the $\Phi_r$ also verifies
  the condition so that it is possible to apply Panchishkin's method in
  order to construct an admissible measure $\Phi^\alpha$

\item or apply Panchishkin's method to the families of the $\Phi_r^\xi$
  (which obviously also verify the level condition) in order to construct
  admissible measures  $\left(\Phi^\xi\right)^\alpha$ and then use them to
  construct $\Phi^\alpha$.
\end{itemize}

In all what follows, in order to simplify the notations, we prove the
result with $\xi =1$ (the trivial character) and abusively write $\Phi_r$
instead of $\Phi_r^1$. After the proof, we will explain why it is
essentially the same for any Dirichlet character $\xi$ modulo $N$. 

\subsection{The divisibility condition}

\begin{prop}
There exists a positive real number $C$ such that for every open set $y +
(p^\nu)$ of $\Zp^\times$ we have

\begin{eqnarray} \label{divisib}\left| U^{2\nu} \sum_{r'=0}^r {r \choose r'} (-y_p)^{r-r'}
  \Phi_{r'}((y)_\nu)) \right|_p \leq C p^{-\nu r}
\end{eqnarray}

\end{prop}

The different steps of the proof are as follow : 
\begin{enumerate}[1)]
\item by definition, the $p$-adic norm of $\sum_{i=0}^r (4\pi y)^{-i}
  \sum_{n=0}^\infty a(i,n) q^n$ is the supremum of the norms of the
  coefficients $|a(i,n)|_p$. To prove the inequality (\ref{divisib}),  it
  is sufficient to prove that the majoration is available for each Fourier
  coefficient. So the first step is to compute these Fourier coefficients,
  and we will show that they are : 
\begin{multline*}
a(i,n) = p^{2\nu i} \sum_{r'=0}^r {r \choose r'} (-y_p)^{r-r'} (-1)^{r'}
(-1)^i {r' \choose i} \frac{\Gamma(k-l-r')}{\Gamma(k-l-r'-i)} \sum_{n_1 +
  n_2 = np^{2\nu}} b_{n_1}  n_2^{r'-i} \\
\times \frac{1}{\varphi(p^\nu)} \sum_{\chi \text{ mod } p^\nu} \chi(n_1)
\bar{\chi}(y) \left(1 - b^{k-l-2r'} \psi \bar{\omega} \bar{\chi}^2
  (b)\right) \sum_{0<d|n_2} \psi \bar{\omega} \bar{\chi}^2(d) d^{k-l-2r'-1}
\end{multline*}

\item by the ultrametric inequality, it is sufficient to prove the
  majoration with fixed $n$, $n_2$ and $d$. Moreover, as $n_2$ is prime to
  $p$ (in the contrary case, the corresponding term is zero), it is the
  same for $d$. Hence, $n_2$ and $d$ are $p$-adic units, so that it is
  sufficient to majorate terms of the form : 
\begin{multline*}
b_{n_1} p^{2\nu i} \sum_{r'=0}^r {r \choose r'} (-y_p)^{r-r'} (-1)^i {r'
  \choose i} \frac{\Gamma(k-l-r')}{\Gamma(k-l-r'-i)}
\frac{1}{\varphi(p^\nu)} \\ \times \sum_{\chi \text{ mod } p^\nu}
\left(-\frac{n_2}{d^2}\right)^{r'} \chi \left(-\frac{n_2}{y d^2} \right)
\left(1-b^{k-l-2r'} \psi \bar{\omega} \bar{\chi}^2(b) \right)
\end{multline*}

\item using Iwasawa's isomorphism (see for example theorem 1.10 of \cite{pant-lnm}), we reduce it to the study of an integral
  of the form $$\int_{\Zp^\times} \chi x_p^{r'} d\mu_b$$ where $\mu_b$ is a
  measure over $\Zp^\times$, only depending on $b$.

\item we introduce a well-chosen differential operator $D$ helping us to
  treat the hypergeometric term $(-1)^i {r' \choose i}
  \frac{\Gamma(k-l-r')}{\Gamma(k-l-r'-i)}$. More explicitly, we will
  prove that : $$\sum_{r'=0}^r {r \choose r'} (-1)^i {r' \choose i}
\frac{\Gamma(k-l-r')}{\Gamma(k-l-r'-i)} z^{r'} = \frac{z^i}{i!} D^i((z-y)_p^r) 
$$

\item After that, it will be easy to prove the congruences we are looking
  for.
\end{enumerate}

\subsubsection{Computation of the Fourier coefficients}

The first step of the proof is to compute the Fourier expansion we will
later majorate. Denote : 
$$\sum_{r'=0}^r {r \choose r'} (-y_p)^{r-r'} \Phi_{r'}((y)_\nu) =
\sum_{i=0}^r (4\pi y)^{-i} \sum_{n=0}^\infty A(i,n) q^n$$

using orthogonality of the characters, we can write : 
$$\Phi_{r'}((y)_\nu) = \frac{1}{\varphi(p^\nu)} \sum_{\chi \text{ mod }
  p^\nu} \bar{\chi}(y) \Phi_{r'}(\chi)$$

However
\begin{eqnarray*}
\Phi_{r'}(\chi) & = & \sum_{y \in p^\nu \Zp^\times} \chi(y) \Phi_{r'}((y)_\nu) \\
& = & (-1)^{r'} \sum_{y,a \in (\mathbb{Z}/p^\nu)^\times} \chi(y) \psi \bar{\omega}(a) g((a^2y)_\nu)
E^b_{r',k-l}((a)_\nu) \\
& = & \left(\sum_{c \in (\mathbb{Z}/p^\nu)^\times} \chi(c) g((c)_\nu) \right) \left( \sum_{a
    \in (\mathbb{Z}/p^\nu)^\times} \psi \bar{\omega} \bar{\chi}^2(a) E^b_{r',k-l}((a)_\nu)\right) \\
& = & g(\chi) E^b_{r',k-l}(\psi \bar{\omega} \bar{\chi^2})
\end{eqnarray*}

Morevover, the Fourier expansions of $g(\chi)$ and $E^b_{r',k-l}(\psi
\bar{\omega} \bar{\chi}^2)$ are known and are : 
$$g(\chi) = \sum_{n=1}^\infty \chi(n) b_n q^n$$
and
\begin{multline*}E^b_{r',k-l}(\psi \bar{\omega} \bar{\chi}^2) =
\varepsilon^b_{r',k-l}(\psi \bar{\omega} \bar{\chi^2}) + \left(1-b^{k-l-2r'}
  \psi \bar{\omega} \bar{\chi}^2(b) \right) (4\pi y)^{-r'} \\
\times \sum_{n=1}^\infty \left( \sum_{0<d\mid n} \psi \bar{\omega} \bar{\chi}^2
  (d) d^{k-l-2r'-1} \right) W(4\pi n y,k-l-r',-r') q^n \end{multline*}

In this case, we previously explicited the Witthaker polynomial which is : 
$$W(y,k-l-r',-r') = \sum_{i=0}^{r'} (-1)^i {r' \choose i}
\frac{\Gamma(k-l-r')}{\Gamma(k-l-r'-i)} y^{r'-i}$$

So we can write the Fourier expansion of $E^b_{r',k-l}(\psi \bar{\omega}
\bar{\chi}^2)$ under the form : 
\begin{multline*}\varepsilon^b_{r',k-l}(\psi \bar{\omega} \bar{\chi}^2) + \left( 1 -
  b^{k-l-2r'} \psi \bar{\omega} \bar{\chi}^2 (b) \right) \sum_{i=0}^{r'}
(4\pi y)^{-i} (-1)^i {r' \choose i}\\
\times \frac{\Gamma(k-l-r'-i)}{\Gamma(k-l-r')}\sum_{n=1}^\infty n^{r'-i} \left(\sum_{0<d\mid n} \psi
  \bar{\omega} \bar{\chi}^2 (d) d^{k-l-2r'-1} \right) q^n \end{multline*}

so that the Fourier expansion of the product $g(\chi) E^b_{r',k-l}(\psi
\bar{\omega} \bar{\chi}^2)$ is

\begin{multline*}
\sum_{i=0}^{r'} (4\pi y)^{-i} (-1)^i {r' \choose i}
\frac{\Gamma(k-l-r')}{\Gamma(k-l-r'-i)}  \left(1 - b^{k-l-2r'} \psi
  \bar{\omega} \bar{\chi}^2 (b) \right) \\ \times \left(\sum_{n_1 + n_2 = n \atop n_2 > 0} \chi(n_1)
b_{n_1} n_2^{r'-i} \left(\sum_{0<d\mid n_2} \psi \bar{\omega} \bar{\chi}^2
  (d) d^{k-l-2r'-1}\right) + \chi(n) b_n \varepsilon^b_{r',k-l}(\psi
\bar{\omega} \bar{\chi}^2) \right)
\end{multline*}

To end the computation of Fourier coefficients, it is sufficient to remind
that $U$ acts on nearly holomorphic modular forms by
$$\left.\left(\sum_{i=0}^r \omega^{-i} \sum_{n=0}^\infty a_i(n) q^n \right) \right| 
U = \sum_{i=0}^r \omega^{-i} p^i \sum_{n=0}^\infty a_i(np) q^n$$

We are now able to compute the coefficients we previously denoted by
$A(i,n)$.

\begin{multline*}
A(i,n) = p^{2\nu i} \sum_{r'=0}^r {r \choose r'} (-y_p)^{r-r'} (-1)^i
{r' \choose i} \frac{\Gamma(k-l-r')}{\Gamma(k-l-r'-i)}(-1)^{r'}
\sum_{n_1+n_2 = np^{2\nu}} b_{n_1} n_2^{r'-i} \\ \times \frac{1}{\varphi(p^\nu)}
\sum_{\chi \text{ mod }p^\nu} \chi(n_1) \bar{\chi}(y) \left(1-b^{k-l-2r'}
  \psi \bar{\omega} \bar{\chi}^2 (b) \right) \sum_{0<d\mid n_2} \psi
\bar{\omega} \bar{\chi}^2 (d) d^{k-l-2r'-1}
\end{multline*}

The non-zero terms in $A(i,n)$ will necessarily be those with $n_1$ prime
to $p$ (in the contrary case, $\chi(n_1)=0$). Then the relations $n_1 +
n_2=np^{2\nu}$ and $d \mid n_2$ show that $n_2$ and $d$ are also prime to
$p$. Then $n_1,n_2$ and $d$ are p-adic units. As a consequence, we can
write $A(i,n)$ as a linear combination with units coefficients of terms of
the following form : 
\begin{multline*}
B(i,n,n_1,d) = p^{2\nu i} \sum_{r'=0}^r {r \choose r'} (-y_p)^{r-r'}
{r' \choose i}(-1)^i \frac{\Gamma(k-l-r')}{\Gamma(k-l-r'-i)} (-1)^{r'}
n_2^{r'} b_{n_1} \\  \times \frac{1}{\varphi(p^\nu)} \sum_{\chi \text{ mod }p^\nu}
\chi(n_1) \bar{\chi}(y) \bar{\chi}^2(d) d^{-2r'} \left(1-b^{k-l-2r'} \psi
  \bar{\omega} \bar{\chi}^2 (b) \right) \\
= p^{2\nu i} \sum_{r'=0}^r {r \choose r'} (-y_p)^{r-r'}
{r' \choose i}(-1)^i \frac{\Gamma(k-l-r')}{\Gamma(k-l-r'-i)}
(-n_2)^{r'} b_{n_1} \\  \times \frac{1}{\varphi(p^\nu)} \sum_{\chi \text{ mod }p^\nu}
\chi(-n_2) \bar{\chi}(y) \bar{\chi}^2(d) d^{-2r'} \left(1-b^{k-l-2r'} \psi
  \bar{\omega} \bar{\chi}^2 (b) \right)
\end{multline*}

\subsubsection{Use of a p-adic integral}
In this section, we wish to use Iwasawa's isomorphism in the goal of
considering integrals over $\Zp^\times$.

Considering the term $1 - b^{k-l-2r'} \psi \bar{\omega} \bar{\chi}^2 (b)$
with $b, \psi$ and $\omega$ fixed, we can see it as a $\Cp$-analytic
bounded function of the variable $x=\chi x_p^{r'}$. Iwasawa's isomorphism
then says that it is the Mellin transform of a measure $\mu$, evaluated at
the point $x$. This measure then verifies : 
$$1-b^{k-l-2r'}\psi \bar{\omega} \bar{\chi}^2(b) = \int_{\Zp^\times} \chi
x_p^{r'} d\mu$$

An important fact in what follows is that this measure only depends on $b$,
$\psi$ and $\omega$ previously fixed, and so is independant of $r'$. Hence,
the following equality holds :

\begin{equationarray*}{lclr}
B(i,n,n_1,d)& =& \multicolumn{2}{l}{b_{n_1} p^{2\nu i} \sum_{r'=0}^r {r \choose r'} (-y_p)^{r-r'} (-1)^i {r' \choose i} \frac{\Gamma(k-l-r')}{\Gamma(k-l-r'-i)}} \\
& & \multicolumn{2}{r}{\times \frac{1}{\varphi(p^\nu)} \sum_{\chi} \chi\left(\frac{-n_2}{yd^2}\right)
\left(\frac{-n_2}{d^2}\right)^{r'} \int_{\Zp^\times} \chi x_p^{r'} d\mu} \\
& = &\multicolumn{2}{l}{ b_{n_1} p^{2\nu i} \sum_{r'=0}^r {r \choose r'} (-y_p)^{r-r'}
(-1)^i {r' \choose i} \frac{\Gamma(k-l-r')}{\Gamma(k-l-r'-i)}} \\
& & \multicolumn{2}{r}{ \times \int_{\Zp^\times}
\frac{1}{\varphi(p^\nu)} \sum_{\chi}\chi\left(\frac{-n_2 x}{ d^2} \right) \left( \frac{-n_2 x}{y
    d^2}\right)_p^{r'} d\mu(x)} \\
& = &\multicolumn{2}{l}{ b_{n_1} p^{2\nu i} \sum_{r'=0}^r {r \choose r'} (-y_p)^{r-r'}
(-1)^i {r' \choose i} \frac{\Gamma(k-l-r')}{\Gamma(k-l-r'-i)}} \\
& &\multicolumn{2}{r}{\times \int_{x
  \equiv -d^2 y n_2^{-1} \text{ mod } p^\nu}
\left(\frac{-n_2 x}{d^2 } \right)_p^{r'} d\mu(x)}
\end{equationarray*}

\subsubsection{Use of a differential operator}
Now, working with $n,n_1$ and $d$ fixed, we wish to introduce a
differential operator which use could help us to understand the origin of
the hypergeometric factor $(-1)^i {r'
  \choose i} \frac{\Gamma(k-l-r')}{\Gamma(k-l-r'-i)}$.

First, let's write
\begin{eqnarray*}{r' \choose i} & = & \frac{r' \times (r'-1) \dots
    (r'-i+1)}{i!} \\
(-1)^i \frac{\Gamma(k-l-r')}{\Gamma(k-l-r'-i)} & = & (-1)^i (k-l-r'-1)
\times \dots \times (k-l-r'-i) \\
& = & (i+r'+l-k) \times \dots \times (1+r'+l-k)
\end{eqnarray*}

Making the change of variable $z = \frac{- x n_2}{d^2}$, we define the
differential operator $D$ by
$$D = z^{k-l-i} \frac{d}{dz} \frac{1}{z^{k-l-i-1}}\frac{d}{dz}$$

It is then easy to check that
$$D z^{r'} = r' (i+r'+l-k) z^{r'-1}$$
and by induction that
$$D^i z^{r'} = r' \dots (r'-i+1) (i+r'+l-k) \dots (1+r'+l-k) z^{r'-i}$$

Consequently, the following equality holds : 
$$\sum_{r'=0}^r {r \choose r'} (-y_p)^{r-r'} (-1)^i {r' \choose i}
\frac{\Gamma(k-l-r')}{\Gamma(k-l-r'-i)} z^{r'} = \frac{z^i}{i!} D^i\left((z-y)_p^r\right)$$

\subsubsection{Proof of the congruences}

Now, we have all the tools we need to prove the inequality
(\ref{divisib}). Remind that by what we already said, it is sufficient to
prove that the inequality holds for $B(i,n,n_1,d)$.

But, using the tools we just introduced, we have : 
\begin{eqnarray*}
B(i,n,n_1,d) & = &b_{n_1}  p^{2\nu i} \sum_{r'=0}^r {r \choose r'}
(-y_p)^{r-r'} (-1)^i {r' \choose i}
\frac{\Gamma(k-l-r')}{\Gamma(k-l-r'-i)} \int_{x \equiv -d^2 y n_2^{-1}}
z^{r'} d\mu(x) \\
& = &\frac{b_{n_1} p^{2\nu i}}{i!} \int_{x \equiv -d^2 y n_2^{-1} \text{
    mod } p^\nu}
D^i\left((z-y)_p^r\right) z^i d\mu(x) \\
& = & \frac{b_{n_1} p^{2\nu i}}{i!} \int_{\Zp^\times} \delta_y(z)
D^i\left((z-y)_p^r \right) z^i d\mu(x)
\end{eqnarray*}
where $\delta_y$ denotes the characteristic function of the open set $y + (p^\nu)$.

Notice that when $z \equiv y \mod p^\nu$, then  $(z-y)^r \equiv 0
\mod p^{\nu r}$ so that $D^i\left((z-y)^r\right) \equiv 0 \mod
p^{\nu(r-2i)}$.
Thereby the following inequality holds : 
\begin{eqnarray*}
\left|B(i,n,n_1,d)\right|_p& \leq & |b_{n_1}|_p |p|_p^{-2\nu i}\frac{1}{|i!|_p} \left| D^i\left((z-y)^r\right)
  y^i \right|_p \\
& \leq & |g|_p \frac{1}{|r!|_p} p^{\nu(2i-2i-r)} \\
& \leq & C p^{-\nu r}
\end{eqnarray*}

This achieves the proof of the divisibility condition in the case where
$\xi$ is the trivial character.\\
Before the proof, we noticed that it would not be really harder with any
character $\xi$ modulo $N$. Indeed, in the upper proof, it suffices to
replace $\chi$ by $\chi \times \xi$, the only change being in the use of
Iwasawa's isomorphism : we still consider a function of the variable $\chi
x_p^r$, which now depends also on $\xi$. This is not embarassing for the
following since we work with $\xi$ fixed. 

\section{Computation of the integrals}

In this section, we try to compute explicitly the integrals obtained with
the admissible measure defined in the previous section (in what follows,
$\Phi^\alpha$ is the distribution constructed previously, not just the cas
$\xi = 1$), and to link them
with Rankin product. For that, the reasonning is analogous to the one of
Panchishkin in  \cite{pant-invent}, where he treats the case when $g$ is an
Eisenstein series convolution. The result is the one given by theorem B.

We want to compute the integrals
\begin{eqnarray} \label{integrale}\int_Y \chi(y) y_p^r l_{f,\alpha}
  (d\tilde{\Phi}^\alpha) \end{eqnarray}
where $\chi$ is a Dirichlet character modulo $p^\nu$ and $0 \leq r \leq k-l-1$.

Notice that the application of the linear form $l_{f,\alpha}$ to the measure
$\tilde{\Phi}^\alpha$ is only in the goal of having numerical
distributions, $\tilde{\Phi}^\alpha$ being a modular distribution.

\begin{eqnarray*}
\int_Y \chi(y) y_p^r l_{f,\alpha}(\tilde{\Phi}^\alpha) & = & l_{f,\alpha}
\left(\sum_{a \in Y_\nu} \chi(a) \int_{(a)_\nu} y_p^r d\tilde{\Phi}^\alpha \right)
  \\
& = & l_{f,\alpha} \left( \sum_{a \in Y_\nu} \chi(a) \Phi_r^\alpha((a)_\nu)
\right)
\end{eqnarray*}

where
$\Phi_r^\alpha((a)_\nu) = U^{-2\nu} \left[ \pi_{\alpha,1} U^{2\nu}
  \Phi_r((a)_\nu) \right]$ denotes the $\alpha$-primary part of $\Phi_r$,
so that

$$
\Phi_r^\alpha((a)_\nu) = U^{-2\nu} \left[ \pi_{\alpha,1} U^{2\nu} \left( (-1)^r
    \sum_{b' \in Y_\nu} \psi \bar{\omega} (b') g((b'^2a)_\nu) E^b_{r,k-l}((b')_\nu)
    \right) \right] 
$$

Hence, 

\begin{eqnarray*}
\int_Y \chi(y) y_p^r l_{f,\alpha} (d\tilde{\Phi}^\alpha) & = & 
 l_{f,\alpha} \left( \sum_{a \in Y_\nu} \chi(a) \Phi_r^\alpha((a)_\nu)
 \right) \\
& = & l_{f,\alpha} \left( U^{-2\nu} \left[ \pi_{\alpha,1} U^{2\nu} \left(
      (-1)^r \sum_{a,b' \in Y_\nu} \chi(a) \psi \bar{\omega}(b')
      g((b'^2a)_\nu) E^b_{r,k-l}((b')_\nu) \right) \right] \right)
\end{eqnarray*}

Since $\chi(a) \psi \bar{\omega} (b') = \chi (ab'^2) \psi \overline{\omega
  \chi^2} (b')$, we have
$$\int_Y \chi(y) y_p^r l_{f,\alpha} (\tilde{\Phi}^\alpha) = l_{f,\alpha}
\left(U^{2\nu} \left[\pi_{\alpha,1} U^{2\nu} \left( (-1)^r g(\chi)
      E^b_{r,k-l}(\psi \overline{\omega \chi^2}) \right) \right] \right)$$

To simplify notations, from now on we denote $h=(-1)^r g(\chi) E^b_{r,k-l}(\psi
\overline{\omega \chi^2})$.

From the definition of $l_{f,\alpha}$
\begin{eqnarray*}
l_{f,\alpha} \left(U^{-2\nu} \left[ \pi_{\alpha,1} U^{2\nu} h \right]
\right) & = & i_p \left( \frac{\left< f^0,\alpha^{-2\nu} U^{2\nu}(h)
    \right>_{Np}}{\left<f^0,f_0\right>_{Np}} \right) \\
& =& i_p \left( \alpha^{-2\nu} p^{2\nu(k-1)} \frac{\left<V^{2\nu}( f^0),h
    \right>_{Np^{2\nu+1}}}{\left<f^0,f_0\right>_{Np}} \right)
\end{eqnarray*}
where $V$ denotes the operator defined by $V(f)(z) = f(pz) = p^{-k/2} f
\mid_k \begin{pmatrix} p & 0 \\ 0& 1\end{pmatrix}$

Last identity is obtained by watching the action of double cosets :
\begin{eqnarray*} <f^0,U^{2\nu}(h)>_{Np} & = & \left< f^0,h\mid_k
    \left[\Gamma_0(Np^{2\nu +1}) \begin{pmatrix} 1 & 0 \\ 0 & p^{2\nu}
      \end{pmatrix} \Gamma_0(Np) \right] \right>_{Np} \\
 & = &  \left< f^0 \mid_k \left[\Gamma_0(Np) \begin{pmatrix} p^{2\nu} & 0
       \\ 0 & 1 \end{pmatrix} \Gamma_0(Np^{2\nu+1}) \right],h
 \right>_{Np^{2\nu +1}}
\end{eqnarray*}

But it is well-known that  $\left[ \Gamma_0(Np) \begin{pmatrix} p^{2\nu} &
    0 \\ 0 & 1 \end{pmatrix} \Gamma_0(Np^{2\nu+1}) \right] = \left[
  \Gamma_0(Np) \begin{pmatrix} p^{2\nu} & 0 \\ 0 & 1 \end{pmatrix} \right]$

so that
$f^0 \mid_k \left[\Gamma_0(Np) \begin{pmatrix} p^{2\nu} & 0 \\ 0 & 1
  \end{pmatrix} \Gamma_0(Np^{2\nu+1}) \right] = p^{2\nu(k-1)}
f^0(p^{2\nu}z) = p^{2\nu(k/2-1)} f^0 \mid_k
\begin{pmatrix} p^{2\nu} & 0 \\ 0 & 1 \end{pmatrix}$

Now, using the action of the principal involution of level $Np^{2\nu +1}$ :$ \begin{pmatrix} 0 & -1 \\ Np^{2\nu +1} & 0 \end{pmatrix}$
$$\left< V^{2\nu} (f^0),h \right>_{Np^{2\nu+1}} = \left<(V^{2\nu}(f^0))
  \mid_k W_{Np^{2\nu+1}} , h \mid_k W_{Np^{2\nu+1}}\right>_{Np^{2\nu+1}}$$

From the definition of $f^0 = f_0^\rho \mid_k W_{Np}$, we have 
\begin{eqnarray*}
(V^{2\nu}(f^0)) \mid_k W_{Np^{2\nu+1}} & = & p^{-2\nu k/2}f^0 \mid_k
\begin{pmatrix} p^{2\nu} & 0 \\ 0 & 1 \end{pmatrix} \begin{pmatrix} 0 & -1
  \\ Np^{2\nu +1} & 0 \end{pmatrix} \\
& = & p^{-2\nu k/2} f_0^\rho \mid_k W_{Np^{2\nu+1}}W_{Np^{2\nu+1}} \\
& = & (-1)^k p^{-\nu k} f_0^\rho
\end{eqnarray*}

so that we transformed the inner product in
$$\left< V^{2\nu}(f^0),h \right>_{Np^{2\nu+1}} = (-1)^k p^{-\nu k} \left<
  f_0^\rho,h \mid_k W_{Np^{2\nu +1}} \right>_{Np^{2\nu+1}}$$

Going back to the computation of the integral (\ref{integrale}) : 
\begin{equationarray*}{l}
(-1)^{k+r} \int_Y \chi(y) y_p^r l_{f,\alpha} (d\tilde{\Phi}^\alpha)  =  
(-1)^{k+r} l_{f,\alpha} \left( U^{-2\nu} \left[\pi_{\alpha,1} U^{2\nu} (h)
  \right] \right) \\ 
{ =  i_p \left( \alpha^{-2\nu} p^{2\nu(k/2-1)} \frac{\left< f_0^\rho,
      g(\chi) \mid_l W_{Np^{2\nu+1}} \times E^b_{r,k-l} (\psi \bar{\omega}
      \bar{\chi}^2) \mid_{k-l} W_{Np^{2\nu+1}}
    \right>_{Np^{2\nu+1}}}{\left<f^0,f_0\right>} \right)}
\end{equationarray*}

From now on, we denote $\omega \chi^2$ the Dirichlet character of $g(\chi)$ by
$\xi$. It remains to express the numerator as a Rankin product. We already know that 
\begin{multline*}
E_{r,k-l}(\psi \bar{\xi})\mid W_{Np^{2\nu+1}} = 
\frac{(Np^{2\nu+1})^{k-l-2r/2} \Gamma(k-l-r)}{(-2i\pi)^{k-l-2r} (-4\pi y)^r} N_0^{-1}
G((\psi \bar{\xi})_0) \\
\times \sum_{0<t \mid N_1} \mu(t) (\psi \bar{\xi})_0(t) t^{-1}
J_{k-l,N_0}(t^{-1}N_1z,-r,\overline{(\psi \bar{\xi})_0})
\end{multline*}

where $(\psi \bar{\xi})_0$ is the primitive Dirichlet character modulo $N$
associated to $\psi \bar{\xi}$ and $N_1 =
  Np^{2\nu+1}/N_0$ and $G((\psi \bar{\chi})_0)$ is its Gauss sum and
  $J_{k-l,N_0}$ is defined as in \cite{pant-invent}. Since

\begin{equationarray*}{l}
E^b_{r,k-l}(\psi \bar{\xi}) = \sum_{a \in Y_\nu} \psi(a)
\bar{\xi}(a) E_{r,k-l}(a) - b^{k-l-2r} \sum_{a \in Y_\nu} \psi(a) \bar{\xi}(a)
E_{r,k-l}(b^{-1}a)\\
 = (1-b^{k-l-2r} \psi(b) \bar{\xi}(b)) E_{r,k-l}(\psi
\bar{\xi})
\end{equationarray*} it is sufficient to treat the case of $E_{r,k-l}$ to deduce
those of $E^b_{r,k-l}$.

In this purpose, we use a result of Shimura, under the form given in
\cite{hida_rankin} :
\begin{prop}
With previous notations, we have
\begin{multline*}
\mathcal{D}_{Np^{2\nu+1}}(l+r,f_0^\rho,g(\chi)\mid W_{Np^{2\nu+1}}) = \\
C \sum_{0<t|N_1} \mu(t)t^{-1} (\psi \bar{\xi})_0
  (t)\left<f_0^\rho,g(\chi)\mid W_{Np^{2\nu+1}}
    J_{k-l,N_0}(t^{-1}N_1z,-r,\overline{(\psi \bar{\xi}
     )_0})y^{-r}\right>_{Np^{2\nu+1}}
\end{multline*}
with $C = \pi^{-k+2l+3r+1} 2^{2l+2r-1} (Np^{2\nu+1})^{k-l-2r-1} N_0^{-1}
G((\psi \bar{\xi})_0) \frac{\Gamma(k-l-r)}{\Gamma(l+r)\Gamma(r+1)}$
where $G((\psi \bar{\xi})_0)$ is the Gauss sum of $(\psi
\bar{\xi})_0$.

\end{prop}

Going back to the inner product appearing in the computation of the
integral :
\begin{multline*}
\left<f_0^\rho,g(\chi) \mid_l W_{Np^{2\nu+1}} E^b_{r,k-l}(\psi \bar{\xi})
  \mid W_{Np^{2\nu+1}} \right> = \\
(1-b^{k-l-2r} \psi \bar{\xi}(b))\frac{(Np^{2\nu+1})^{\frac{k-l-2r}{2}} \Gamma(k-l-r)}{(-2i\pi)^{k-l-2r}
  (-4\pi)^r} N_0^{-1} G((\psi \bar{\xi})_0) \\
\times \sum_{0<t\mid N_1} \mu(t) (\psi \bar{\xi})_0(t) t^{-1}
\left<f_0^\rho,g(\chi)\mid W_{Np^{2\nu+1}} J_{k-l,N_0}(t^{-1} N_1 z,
  -r,\psi \bar{\xi}) y^{-r} \right>_{Np^{2\nu+1}} 
\end{multline*}

Using Shimura's formula, this is equal to :
$$
\left(1-b^{k-l-2r} \psi \bar{\omega}(b) \right)\frac{(Np^{2\nu+1})^\frac{k-l-2r}{2} \Gamma(k-l-r)}{(-2i\pi)^{k-l-2r} (-4
  \pi)^r} N_0^{-1} G((\psi \bar{\xi})_0)
  \mathcal{D}_{Np^{2\nu+1}}(l+r,f_0^\rho,g(\chi)\mid_l W_{Np^{2\nu+1}})
  C^{-1}$$

Thus the integral (\ref{integrale}) is equal to :
\begin{multline*}
(-1)^{k+r} \int_Y \chi(y) y_p^r d\tilde{\Phi}^\alpha(y) = \alpha^{-2\nu}
p^{2\nu(k/2-1)}(1-b^{k-l-2r} \psi \bar{\xi}(b)) \frac{\mathcal{D}_{Np^{2\nu+1}}(l+r,f^\rho_0,g(\chi)\mid_l
  W_{Np^{2\nu+1}})}{\left< f^0,f_0\right>_{Np}} \\
\times \frac{(Np^{2\nu+1})^{\frac{k-l-2r}{2}}
  \Gamma(k-l-r)}{(-2i\pi)^{k-l-2r} (-4\pi)^r} N_0^{-1} G((\psi
\bar{\xi})_0) \pi^{k-2l-3r-1} 2^{1-2l-4r} N^{1+2r+l-k} \\
\times N_0 G((\psi \bar{\xi})_0)^{-1} p^{(1+2r+l-k)(2\nu+1)} \frac{\Gamma(l+r)
  \Gamma(r+1)}{\Gamma(k-l-r)} 
\end{multline*}

Grouping similar terms, we  have
\begin{multline*}
(-1)^{k+r} l_{f,\alpha} \left(\int_Y \chi(y) y_p^r d\tilde{\Phi}^\alpha(y)\right) = \alpha^{-2\nu}
(1-b^{k-l-2r} \psi \bar{\chi}(b)) \frac{\mathcal{D}_{Np^{2\nu+1}}(l+r,f_0^\rho,g(\chi) \mid_l
  W_{Np^{2\nu+1}})}{\left<f^0,f_0\right>_{Np}} \\
\times \pi^{-l-2r-1} 2^{1-k-l-2r} N^{1-(k-l-2r)/2}p^{(2\nu+1)(1-(k-l-2r)/2)} (-1)^r i^{k-l-1}
\Gamma(l+r) \Gamma(r+1)
\end{multline*}

In order to conclude the proof of the theorem, let us use  the unicity
property of admissible measures (Thm \ref{unicite-admissible}). This proves that the admissible measure we constructed is the only admissible
measure taking those values in the critical strip $\{l,l+1, \dots,k-1\}$.

\vspace{20pt}

\noindent\textbf{Acknowledgements:} I thank Alexei Panshichkin for having
suggested me this work. I also thank Fran\c cois Brunault for his careful
reading and helpful comments.

 \end{document}